\documentclass[11pt, a4paper]{amsart}

\usepackage{dsfont}
\usepackage{amsmath, amsthm, amssymb}
\usepackage{xypic}
\usepackage{graphicx}
\usepackage{color}
\usepackage[usenames,dvipsnames]{xcolor}
\usepackage[latin1]{inputenc}
\usepackage{indentfirst}
\usepackage{epstopdf}
\usepackage{subfigure}
\usepackage{emptypage}
\usepackage{hyperref}
\hypersetup{
    colorlinks=true, 
    linktoc=all,     
    linkcolor=blue,  
    citecolor=JungleGreen,
}

\newtheorem{theorem}{Theorem}
\newtheorem{corollary}[theorem]{Corollary}
\newtheorem{example}[theorem]{Example}
\newtheorem{proposition}[theorem]{Proposition}
\newtheorem{lemma}[theorem]{Lemma}
\newtheorem{definition}[theorem]{Definition}
\newtheorem{remark}[theorem]{Remark}

\newtheorem{theorem*}{Theorem}
\newtheorem{question*}[theorem*]{Question}
\newtheorem{conjecture*}[theorem*]{Conjecture}
\newtheorem{corollary*}[theorem*]{Corollary}
\newtheorem{theorem*e}{Teorema}
\newtheorem{question*e}[theorem*e]{Pregunta}
\newtheorem{conjecture*e}[theorem*e]{Conjetura}
\newtheorem{corollary*e}[theorem*e]{Corolario}

\newcommand{\Fix}{\mathrm{Fix}}
\newcommand{\Per}{\mathrm{Per}}

\renewcommand{\int}{\mathrm{int}}

\renewcommand{\deg}{\mathrm{deg}}
\renewcommand{\det}{\mathrm{det}}
\newcommand{\id}{\mathrm{id}}
\newcommand{\dist}{\mathrm{dist}}

\newcommand{\R}{\mathds{R}}
\newcommand{\Z}{\mathds{Z}}
\newcommand{\z}{\mathcal{Z}}
\newcommand{\N}{\mathds{N}}
\newcommand{\Q}{\mathds{Q}}
\newcommand{\C}{\mathds{C}}
\newcommand{\e}[1]{\overset{(#1)}{=}}

\title{Uniqueness of dynamical zeta functions and symmetric products}

\author[E. Blanco Gómez]{Eduardo Blanco Gómez}
\address{Facultad de Matemáticas UCM, Plaza de Ciencias 3, Madrid, Spain.}
\email{eduardo.blanco.gomez@gmail.com}

\author[L. Hernández--Corbato]{Luis Hernández--Corbato}
\address{IMPA, Estrada dona Castorina 110, Rio de Janeiro, Brazil.}
\email{luishcorbato@mat.ucm.es, lcorbato@impa.br}

\author[F. R. Ruiz del Portal]{Francisco R. Ruiz del Portal}
\address{Facultad de Matemáticas UCM, Plaza de Ciencias 3, Madrid, Spain.}
\email{rrportal@ucm.es}

\thanks{The authors have been supported by MINECO, MTM2012-30719.
The second author has also been supported by CNPq.}

\begin{document}

\begin{abstract}
A characterization of dynamically defined zeta functions is presented.
It comprises a list of axioms, natural extension of the one which characterizes topological degree, and a uniqueness theorem.
Lefschetz zeta function is the main (and proved unique) example of such zeta functions.
Another interpretation of this function arises from the notion of symmetric product
from which some corollaries and applications are obtained.

\end{abstract}

\maketitle
\section{Introduction}

Solving equations or counting the number of zeros of a function has been one the major tasks of
mathematics since its very beginning. In various settings, this number only gains relevance when
it is robust, i. e., it is not affected by small perturbations of the map or equation. Clearly,
the number of zeros of $f(x) = x^2 + c$ is not robust as $c$ goes through 0. Both solutions of $x^2 + c = 0$
for $c < 0$ collapse into one and then vanish as $c$ goes positive.
The topological degree assigns multiplicities to each zero in a way that their sum is preserved.

After the contribution of many of the most renowned mathematicians, it was Hopf who settled the concept
of topological degree, which he called Brouwer--Kronecker degree as their works were crucial in the foundation of this theory.
Leray and Schauder extended the definition of degree from polyhedra to Banach spaces,
Nagumo proposed an axiomatization together with an uniqueness result which was proved
independently by F\"{u}hrer \cite{fuhrer} in the Euclidean case and by Amann and Weiss
\cite{amannweiss} in a more general setting.

There are natural restrictions on the maps to which a degree is associated. For example, the case
in which a zero of a map is located in the boundary of its domain can not be handled properly.
In $\R^d$, admissible maps
$f: \overline{U} \to \R^d$ are those continuous maps defined in the closure of an open subset $U$ of $\R^d$
such that $f^{-1}(0) \cap \partial U = \emptyset$ and $f^{-1}(0) \cap U$ is compact.
The axiomatization and uniqueness, stated in Subsection \ref{subsec:z1istheindex},
shows that the topological degree is characterized by the value
given to the identity map, an additivity property for disjoint domains and homotopy invariance.

The direct translation of the concept of degree to account for fixed points is called fixed point index.
The index of a map $f$ is defined as the degree of the map $\id - f$. Given an open subset
$U$ of $\R^d$ and a continuous map $f: \overline{U} \to \R^d$, the fixed point index $i(f, U)$ of $f$ in $U$
is an integer which is well--defined as long as $\Fix(f) \cap \overline{U} = \Fix(f) \cap U$
and this set is compact. The index can be extended to Euclidean neighborhood retracts (ENRs).
The reader is referred to \cite{doldindice} for the definition and to \cite{granas, JM} for a complete account on 
fixed point index.
This invariant counts the number of fixed points of $f$ in $U$ with multiplicity.
It makes sense to look at the iterates of $f$, denoted $f^n$. The value $i(f^n, U)$
depends on the points periodic under $f$ and whose period is divisible by $n$. The sequence $(i(f^n, U))_{n \ge 1}$
is encoded in the so--called Lefschetz or homological zeta function
$$\exp \left( \sum_{n \ge 1} \frac{i(f^n, U)}{n} \cdot t^n \right).$$
Lefschetz--Hopf fixed point theorem, where applicable, guarantees this formal power series to be rational.
This is the case, for example, when $U$ is a closed manifold or a compact ENR.
Rationality is a consequence of the fixed point indices being homological invariants of $f$.
Note however that if the index $i(f^n)$ is replaced by the number $\#\Fix(f^n)$ of fixed points of $f^n$
the resulting zeta function, typically named after Artin and Mazur, is not rational in general in the absence
of hyperbolicity.

In this paper, an axiomatic characterization of the Lefschetz zeta function is given.
Both the statement and the proof of the uniqueness go along the lines of aforementioned results for the topological degree.
Let $U$ be an open subset of $\R^d$ and $f: \overline{U} \to \R^d$ continuous.
The pair $(f, U)$ is admissible provided that, for every $k \ge 1$,
$\Fix(f^k) \cap \overline{U} = \Fix(f^k) \cap U$ and it is compact.
Denote $\mathcal{A}(U)$ the set of all admissible pairs and
$1 + t \cdot \Z [[t]]$ the group of formal monic power series with integer coefficients.
The main result of the article reads as follows.

\begin{theorem}\label{thm:main}
There exist just one function 
$$\z: \mathcal{A}(U) \to 1 + t \cdot \Z [[t]]$$
which satisfies the following properties:
\begin{itemize}
\item[(N)] \textbf{Normalization.} The value for the constant map $c_p : \overline{U} \to \R^d$
which sends every point to $p \in U$ is
$$\z(c_p, U)(t) = 1 / (1 - t).$$
\item[(M)] \textbf{Multiplicativity.} Let $(f, U)$ be an admissible pair and $V, W$ disjoint
open subsets of $U$ such that $\Per(f) \cap U = \Per(f) \cap (V \cup W)$ and $\Per(f) \cap V, \Per(f)\cap W$
are invariant under $f$. Then,
$$\z(f, U) = \z(f_{|V}, V)\z(f_{|W}, W).$$
\item[(H)] \textbf{Homotopy invariance.} If $\{(h_t, U)\}_{t=0}^1$ is an admissible homotopy then
$$\z(h_0, U) = \z(h_1, U).$$
\item[(I)] \textbf{Iteration.} Let $(f, U)$ be an admissible pair and assume that $U$ is the disjoint union
of $k \ge 1$, $U = U_1 \cup \ldots \cup U_{k}$, such that $f(\Per(f) \cap U_i) \subset U_{i+1}$ (indices are taken mod $k$). Then
$$\z(f, U)(t) = \z((f^k)_{|U_1}, U_1)(t^k).$$
\end{itemize}
\end{theorem}

The axioms are analogous to the ones of the topological degree except for the last new one.
It basically establishes how the power series integrates the information of periodic orbits.
A function satisfying (N), (M), (H) and (I) is called a dynamical zeta function.

The $n^{th}$ symmetric product $SP_n(X)$ of a topological space $X$ is the quotient of $X^n$
by the action of the $\Sigma_n$ group of permutations of $n$ elements. The image of $(x_1, \ldots, x_n)$ under
the projection map is the unordered $n$--tuple $[x_1, \ldots, x_n]$.
A continuous map $f: U \subset X \to X$ induces a map $SP_n(f): SP_n(U) \to SP_n(X)$ which sends
$[x_1, \ldots, x_n]$ to $[f(x_1), \ldots, f(x_n)]$. Note that if $x = f^n(x)$ is a periodic point then
$[x, f(x), \ldots, f^{n-1}(x)]$ is a fixed point of $SP_n(f)$.
Thus, it is possible to study the periodic orbits of $f$ using the maps $SP_n(f)$ instead of $f^n$.

The topological study of symmetric products and its mere description was subject of many articles mostly during
the middle decades of last century, see for instance \cite{BU, doldhomologiaSP, doldthom, macdonald}.
Topological properties are typically inherited, for example $SP_n(X)$ is a ENR provided that $X$ is also a ENR.
There are several papers in the literature where the authors (see \cite{Gor, Mas, Mas1, Ra})
study a fixed point index in symmetric products but from the point of view of multivalued maps,
i.e. for maps $F:X \rightarrow SP_n(X)$.
A fixed point index in hyperspaces which exploits
the good local properties of the hyperspace of polyhedra (see \cite{Cu78, CuSc}) is introduced in \cite{RuSaII}.
In the same direction, in \cite{Sal2}, the author proposes a construction similar to ours but
working in the spaces $F_n(X)$ of finite sets of $X$ with at most $n$ elements.
The main problem of that approach is the absence of an additivity property which makes computations very difficult.

For any admissible pair $(f, U)$
the integer sequence $(i(SP_n(f), SP_n(U)))_{n \ge 1}$ is well--defined. These indices fit in the preceding setting as follows.

\begin{proposition}\label{prop:iszeta}
$SP_{\infty}: \mathcal{A}(U) \to 1 + t \cdot \Z [[t]]$ defined by
$$SP_{\infty}(f, U) = 1 + \sum_{n \ge 1} i(SP_n(f), SP_n(U)) \cdot t^n$$
is a dynamical zeta function, i.e, 
satisfies axioms (N), (M), (H) and (I) in Theorem \ref{thm:main}.
\end{proposition}

From the uniqueness we recover a result of Dold \cite{doldparma}: $SP_{\infty}$ is equal to the Lefschetz zeta function.
The authors have learnt that a brief sketch of an allegedly independent proof of this equality
appeared in a work of Salamon \cite{salamon}.
Besides, as a corollary, already included in Dold's work, a formula due to Macdonald \cite{macdonald}
for the Euler characteristic of $SP_n(X)$ is deduced.

%
%

The paper is organized as follows. Theorem \ref{thm:main} is proved in Section \ref{sec:unicidadzeta},
where the abstract notion of dynamical zeta function, a function satisfying the axioms of Theorem \ref{thm:main},
is introduced. Section \ref{sec:maszeta} explores more properties of this concept and relates it to
the fixed point index by proving the Lefschetz zeta function satisfies the axioms.
Symmetric products and their corresponding zeta function are the content of Section \ref{sec:productossimetricos}.
Some corollaries of the results are shown in Section \ref{sec:applications} along with a digression concerning
planar dynamics. Finally, there is an appendix containing notation and basic results of power series and combinatorics.

\section{Uniqueness of the dynamical zeta function}\label{sec:unicidadzeta}

\subsection{Axioms}
Given an open subset $U$ of $\R^d$  and a map $f : \overline{U} \to \R^d$,
the pair $(f, U)$ is called admissible provided that, for every $k \ge 1$,
$\Fix(f^k) \cap \overline{U} = \Fix(f^k) \cap U$ and it is compact.
A continuous family of maps $h_t : U_0 \to \R^d$, $t \in [0,1]$, forms an admissible homotopy
$\{(h_t, U)\}_{t=0}^1$ if
$$\bigcup_{t \in [0,1]} \{(t, x): x \in \Fix(h_t^k) \cap \overline{U}\}$$
is compact and it is contained in $[0,1] \times U$ for every $k \ge 1$.

Notice that the previous definitions do not assume compactness of the set of periodic points,
denoted $\Per(f)$, within $U$. Merely compactness of $\Fix(f^k)$,
for every $k \ge 1$, or, equivalently, compactness of $\Per_m(f)$, for every $m \ge 1$, is needed.
$\Per_m(f)$ denotes the set of periodic points whose period is bounded by $m$.
A similar remark applies to homotopies.

The set of all admissible pairs $(f, U)$ for a fixed open set $U \subset \R^d$ is denoted $\mathcal{A}(U)$.
A dynamical zeta function $\z$ is a map
$$\z: \mathcal{A}(U) \to 1 + t \cdot \Z [[t]]$$
which satisfies
all four axioms in Theorem \ref{thm:main}: (N) Normalization, (M) Multiplicativity, (H) Homotopy invariance and (I) Iteration.

\begin{remark}
The admissibility of $(f, U)$ automatically implies
the pairs $(f_{|V}, V), (f_{|W}, W)$ in axiom (M) and $((f^k)_{|U_1}, U_1)$ in axiom (I) are admissible as well.
\end{remark}

Let us extract some information from the definition of $\z$. Since the pairs $(f, \emptyset)$ and $(f, \emptyset)$ are admissible
and their domains have empty intersection, Multiplicativity (M) gives
$$\z(f, \emptyset) = \z(f, \emptyset)^2,$$
hence $\z(f, \emptyset) = 1$. From this easy calculation we deduce an important property:

(L) \textbf{Localization}. Assume $(f, U)$ is admissible and $V$ is an open subset of $U$ which contains all
periodic points of $f$. Then,
$$\z(f_{|V}, V) = \z(f, U).$$

\begin{proof}
It is direct application of (M) for $W = \emptyset$.
\end{proof}

\begin{remark}
It is evident from (M) that multiplicativity also holds if the periodic point set is divided into finitely many
invariant pieces.
\end{remark}


\subsubsection{$n$--admissibility.}

The notion of admissibility of pairs is probably too restrictive for our goals. It is typically impracticable
to control all periodic orbits of a map. However, restrictions on $\Fix(f)$
or even on $\Per_n(f)$ for large $n$ are possible to handle with. Thus, some kind
of partial admissibility is needed.

Let $n$ be a positive integer, $U$ an open subset of $\R^d$ and
$f: \overline{U} \to \R^d$ a continuous map. The pair $(f, U)$ is said to be $n$--admissible if
$\Fix(f^k) \cap \overline{U} = \Fix(f^k) \cap U$ and it is compact for every $1 \le k \le n$.
Equivalently, $\Per_n(f) \cap \overline{U} = \Per_n(f) \cap U$ and it is compact.
Similarly, a homotopy $\{(h_t, U)\}_{t = 0}^1$ is $n$--admissible if the conditions for admissibility are
fulfilled for every $1 \le k \le n$ or, likewise, all pairs $(h_t, U)$ are $n$--admissible.
The set of $n$--admissible pairs is denoted $\mathcal{A}_n(U)$.

It is convenient to explicitly state one very trivial remark concerning the previous definition.
\begin{lemma}
A pair (resp. a homotopy) is admissible if and only if it is $n$--admissible for every $n \ge 1$.
\end{lemma}

The axiomatic definition of dynamical zeta function can be extended to $n$--admissible pairs. However, the price to be paid
is that the formal power series is defined only up to the term $t^n$, that is, the
subgroup $1 + t^{n+1} \cdot \Z[[t]]$ must be quotiented out. More precisely,

$$\z_n: \mathcal{A}_n(U) \to (1 + t \cdot \Z [[t]]) \, / \, (1 + t^{n+1} \cdot \Z[[t]]).$$

The set of equivalence classes on the right inherits a group structure with multiplication.
The class to which $z(t)$ belongs is denoted $z(t) \mod t^{n+1}$.
\begin{lemma}
Two formal power series $z, z'$ are equal if and only if $z \mod t^n = z' \mod t^n$ for every $n \ge 1$.
\end{lemma}
This simple fact will be constantly used through this work to obtain the zeta function of an admissible pair $(f, U)$
from the sequence of values $\z_n(f, U)$.

For any positive integers $k, m, n$ with $n \le mk$ there is a well--defined homomorphism
\begin{equation}\label{eq:zetatk}
  r_{m,n}^k: z(t) \mod t^{m+1} \mapsto z(t^k) \mod t^{n+1}
\end{equation}
between the groups $(1 + t \cdot \Z [[t]]) \, / \, (1 + t^{m+1} \cdot \Z[[t]])$ and
$(1 + t \cdot \Z [[t]]) \, / \, (1 + t^{n+1} \cdot \Z[[t]])$.

For the sake of completeness
we list the axioms which define the dynamical zeta function for $n$--admissible pairs.
\begin{itemize}
\item[(N)$_n$] \textbf{Normalization.} The value for the constant map $c_p : U \to \R^d$ which sends every point to $p \in U$ is
$$\z_n(c_p, U)(t) = 1 / (1 - t) \mod t^{n+1} = 1 + t + \ldots + t^n \mod t^{n+1}.$$
\item[(M)$_n$] \textbf{Multiplicativity.} Let $(f, U)$ be an $n$--admissible pair and $V, W$ disjoint
open subsets of $U$ such that $\Per_n(f) \cap U = \Per_n(f) \cap (V \cup W)$ and $\Per_n(f) \cap V, \Per_n(f)\cap W$
are invariant under $f$. Then,
$$\z_n(f, U) = \z_n(f_{|V}, V)\z_n(f_{|W}, W).$$
\item[(H)$_n$] \textbf{Homotopy invariance.} If $\{(h_t, U)\}_{t=0}^1$ is an $n$--admissible homotopy then
$$\z_n(h_0, U) = \z_n(h_1, U).$$
\item[(I)$_n$] \textbf{Iteration.} Let $(f, U)$ be an $n$--admissible pair and assume that $U$ is the disjoint union
of $k \ge 1$, $U = U_1 \cup \ldots \cup U_{k}$, such that $f(\Per_n(f) \cap U_i) \subset U_{i+1}$ (indices are taken mod $k$).
If $m = \lceil \frac{n}{k} \rceil$ then
$$\z_n(f, U)(t) = \z_m((f^k)_{|U_1}, U_1)(t^k),$$
\end{itemize}
where the term on the right has to be understood as $r_{m,n}^k(\z_m((f^k)_{|U_1}, U_1))$ using the notation in (\ref{eq:zetatk}).
Note from the assumptions that $((f^k)_{|U_1}, U_1)$ is automatically $m$--admissible.



\subsection{Computation for hyperbolic linear maps}\label{subsec:hyperbolic}

Let us now compute the zeta function of hyperbolic linear maps directly from the axioms. Since we have full control
of the periodic point set there is no need to work with partial admissibility.

\subsubsection{One--dimensional maps}\label{subsubsec:1d}
Let $f(x) = ax + b$, where $a \neq -1, 1$.
Assume first that $|a| < 1$. Then, the homotopy
$$h_t(x) = (1 - t) ax + b, \enskip \enskip 0 \le t \le 1,$$
is admissible and connects the map $f$ to the constant map $c_b$ so
$$\z(f, \R) \e{H} \z(c_b, \R) \e{N} 1/(1 - t).$$

Assume now that $a > 1$ and, after a change of coordinates, that $b = 0$.
Define $g : \R \to \R$ as $f$ in $(- \infty, 1]$ and as the constant map with value $a$ in $[1, +\infty)$. Then,
$$\z(g, \R) \e{L} \z(g, \R \setminus \{1\}) \e{M} 
\z(f, (-\infty, 1)) \z(c_a, (1, +\infty)) \e{L} \z(f, \R) \z(c_a, \R).$$
On the other hand, one can check that the homotopy $h_t(x) = g(x) - ta$ is admissible.
Note further that $h_1$ has no periodic points, hence we obtain
$$\z(g, \R) \e{H} \z(h_1, \R) \e{L} \z(h_1, \emptyset) = 1.$$
Combining the two expressions yields
$$\z(f, \R) = (\z(c_a, \R))^{-1} \e{N} 1 - t.$$

Last, we examine the case $a < -1$. Assume again $b = 0$ and define $g: \R \to \R$ by
$g(x) = -a$ if $x \le -1$, $g(x) = f(x)$ if $-1 \le x \le 1$ and $g(x) = a$ if $x \ge 1$. In terms of periodic points,
$g$ only has a fixed point at 0 and a 2--periodic orbit $\{a, -a\}$ outside $[-1,1]$. Thus,
\begin{equation}\label{eq:dim1}
\z(g, \R) \overset{(L)+(M)}{=} \z(g_{|V}, V) \z(f, (-1,1)),
\end{equation}
where $V = V^- \cup V^+ = (-\infty, -1) \cup (1, +\infty)$. Note that $g^2_{|V^+} : (1, +\infty) \to (1, +\infty)$
is equal to $c_{-a}$ restricted to $V^+$, so we have
$$\z(g_{|V}, V)(t) \e{I} \z(g^2_{V^+}, V^+)(t^2) = \z(c_{-a}, V^+)(t^2) \e{L} \z(c_{-a}, \R)(t^2) = 1/(1 - t^2).$$
Now, define the homotopy $h_t(x) = (1-t) g(x)$ and check that $\{(h_t, \R)\}_{t = 0}^1$ is admissible.
Indeed, $\Per(h_t) = \{0, (1-t)a, -(1-t)a\}$. Therefore, we obtain $\z(g, \R) \e{H} \z(c_0, \R) = 1/(1-t)$.
Plugging these results into (\ref{eq:dim1}) we obtain
$$\z(f, \R) \e{L} \z(f, (-1,1)) = \frac{1}{1-t} \left(\frac{1}{1-t^2}\right)^{-1} = 1 + t.$$

\subsubsection{Higher dimensions}\label{subsubsec:higherlinear}
There is not a significant gap between $\R$ and $\R^d$ for $d \ge 2$.
However, we will now deal with several eigenvalues and the computations become a bit trickier. A $d \times d$ real matrix $A$
is hyperbolic provided $\sigma(A) \cap S^1 = \emptyset$, $\sigma(A)$ denotes the set of eigenvalues of $A$
with multiplicities.
Actually, a weaker hypothesis than hyperbolicity is required, we simply assume that no eigenvalue of $A$
is a root of unity.
By abuse of notation, let us use $A : \R^d \to \R^d$ to refer to the linear
transformation defined by the matrix $A$. Notice that $\Per(A) = \{0\}$.

We will prove that the zeta function of $A$ only depends on the numbers
$$\sigma^- = \# \{\sigma(A) \cap (-\infty, -1)\} \mod 2, \enskip \enskip \sigma^+ = \# \{\sigma(A) \cap (1, +\infty)\} \mod 2.$$
Thus, only four different behaviors for the zeta function appear, three of which have already been found in dimension 1.

The first step is to prove that $\z(A, \R^d)$ only depends on the set of eigenvalues of $A$.
Henceforth we omit the domain from the notation of the zeta function as long as it is clear from the context.
Let $J_A = Q^{-1} A Q$ be the Jordan canonical form of $A$ and $P : [0,1] \to GL(d, \R)$ such that $P(0) = I$
and $P(1) = Q$. Then, $\{P(t)^{-1} A P(t)\}_{t=0}^1$ is an admissible homotopy from $A$ to $J_A$.
Denote $D_A$ the matrix obtained from $J_A$ after removing all of its 1's. It is composed of non--zero
$1 \times 1$ and $2 \times 2$ blocks on the diagonal which correspond to real and complex eigenvalues of $A$, respectively.
The convex combination  $(1 - t) J_A + t D_A$ defines an admissible homotopy from $J_A$ to $D_A$.
Therefore, after admissible homotopies we can suppose that $A$ has the form of $D_A$.

It is fairly straightforward to show by moving eigenvalues along the real line that we can assume that
any real eigenvalue of $A$ belongs to $\{-2, 0, 2\}$.
Suppose now that $\lambda_0$ is a non--real eigenvalue of $A$, then so is $\bar{\lambda}_0$.
There is a path $\lambda: [0,1] \to \mathds{C}$ connecting $\lambda(0) = \lambda_0$ to $\lambda(1) = 0$
that avoids points of the form $e^{2\pi i\alpha}$ for $\alpha \in \mathds{Q}$.
Likewise, the path $\bar{\lambda}$ connects $\bar{\lambda}_0$ to 0. Therefore, there is a path of
diagonal real matrices which defines an admissible homotopy that makes the pair of eigenvalues $\{\lambda_0, \bar{\lambda}_0\}$
become the pair $\{0, 0\}$. This trick can be used to cancel out every pair of non--real eigenvalues and,
furthermore, the same procedure works for every pair of equal real eigenvalues as well.
The final outcome is a matrix which has at most one eigenvalue in $(-\infty, -1)$, at most one eigenvalue
in $(1, +\infty)$ and no other non--zero eigenvalue.

Up to rearrangement of the columns, which can be simply done through an admissible homotopy for the basis,
for any fixed dimension $d$ only four different matrices appear and they
correspond with the four possibilities $(\sigma^-, \sigma^+) \in \{(0,0),(0,1),(1,0),(1,1)\}$.
Accordingly, denote the matrices $A_{(0,0)}, A_{(0,1)}, A_{(1,0)}, A_{(1, 1)}$.
Note that $A_{(0,0)}$ correspond to the constant linear map
so by (N) we have that $\z(A_{(0,0)}) = 1/(1 - t)$. The behavior of the next two matrices
$A_{(0,1)}, A_{(1,0)}$ is genuinely one--dimensional as they only have one non--zero eigenvalue.
Indeed, adding the constant map in the extra dimensions of the maps of homotopies defined in \ref{subsubsec:1d}
for the cases $a > 1$ and $a < -1$ suffices to compute the zeta function of $A_{(0,1)}$ and $A_{(1,0)}$, respectively.
It follows that $\z(A_{(0,1)}) = 1 - t$ and $\z(A_{(1,0)}) = 1 + t$.

It remains to address the case $(\sigma^-, \sigma^+) = (1,1)$, which corresponds to a true 2--dimensional setting.
The admissible homotopy suggested in Figure \ref{fig:homotopy} serves to this purpose.
Arrows are placed to describe the action of the map, bearing in mind that it also involves a reflection in the
horizontal axis.
We start with a diffeomorphism in $\R^2$ whose dynamics is described on the left of the figure.
It has only two periodic points, which are both fixed, and the dynamics
around them are, up to translation, equal to $A_{(1,0)}$ and $A_{(1,1)}$, respectively. The homotopy collapses both
fixed points into one (middle figure) and then removes it (right figure). After suitably applying (L), (M) and (H) we obtain
$$\z(A_{(1,0)}) \z(A_{(1,1)}) = 1, \enskip \Rightarrow \enskip \z(A_{(1,1)}) = 1/(1 + t).$$
Again, this computation can be extended to arbitrary $d > 2$ by putting the constant map in the extra dimensions.

\begin{figure}[htb]
\begin{center}
\includegraphics[scale = 0.7]{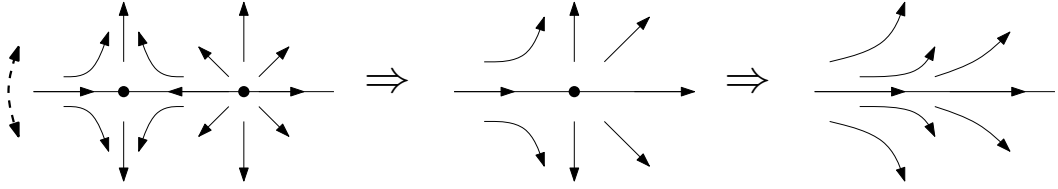}
\end{center}
\caption{Description of the admissible homotopy introduced to compute $\z(A_{(1,1)})$.}
\label{fig:homotopy}
\end{figure}

\subsubsection{$n$--admissibility and hyperbolicity}

The linear map $A$ is $n$--admissible iff $\Per_n(A) = \{0\}$ or, equivalently, if no root
of $A$, viewed as a matrix, is a root of unity of order no greater than $n$.
The homotopies used in \ref{subsubsec:higherlinear} are $n$--admissible as long as the original
map is $n$--admissible.
Consequently, the zeta function
for these linear maps depend on the spectrum of $A$ in the same fashion as in \ref{subsubsec:higherlinear}
with the exception that they are just defined up to the term $t^{n}$.


For example, consider the map $f(x) = -x + b$. Since $-1$ is a $2^{nd}$--root of unity, the zeta function of $f$
is defined only modulo $t^2$. It is easy to check that the first homotopy in \ref{subsubsec:1d}
is 1--admissible and we conclude by (H) that
$$\z_1(f, \R) = 1/(1 - t) \mod t^2 = 1+t \mod t^2.$$

\subsection{Linearization}

Our task is now to prove that the linear approximation is enough to compute the zeta function locally.

\begin{proposition}\label{prop:linearization}[Linearization]
Let $f : V \to \R^d$ be a $C^1$ map and $p$ be a fixed point of $f$. Assume that $V \cap \Per_n(f) = \{p\}$
and $Df_p$, the differential of $f$ at $p$, does not have eigenvalues of the form $e^{2\pi i \alpha}$
for $\alpha = r / q$ with $\gcd(r, q) = 1$ and $q \le n$. Then,
$$\z_n(f, V) = \z_n(Df_p, \R^d).$$
Consequently, if $V \cap \Per_n(f) = \{p\}$ for every $n \ge 1$ and $\sigma(Df_p)$ does not contain roots of unity then
$$\z(f, V) = \z(Df_p, \R^d).$$
\end{proposition}

\begin{proof}
For simplicity assume $p$ is the origin and denote $A = Df_p$.
We will prove that the segment homotopy between $f$ and $A$ is $n$--admissible in a neighborhood of $p$.
There is a map $\phi: V \to \R^d$ such that
$f(x) = Ax + \phi(x)$ for every $x \in V$ and $\lim_{x \to 0} \frac{\phi(x)}{||x||} = 0$.
Define $h_t(x) = Ax + t \phi(x) = (1-t)Ax + tf(x)$, for $t \in [0, 1]$.

By hypothesis, there exists $\epsilon > 0$ such that $||(I - A)x|| > \epsilon||x||$ for every $x$.
On the other hand, there exists $W$ a neighborhood of $p$ such that $||\phi(x)|| \le \frac{\epsilon}{2} ||x||$
for every $x \in W$. Thus, if $x \in W$ we obtain
$$||x - h_t(x)|| = ||(I - A)x + t \phi(x)|| \ge ||\epsilon x - t \frac{\epsilon}{2} x|| > \frac{\epsilon}{2} ||x||.$$
In particular, $p$ is the only fixed point of $h_t$ in $W$ so $\{(h_t, W)\}$ is 1--admissible.

Similar estimates for iterates of $f$ and $A$ are needed to prove the proposition.
Take $\epsilon' > 0$ so that $||(I - A^k)x|| > \epsilon'||x||$ for every $x$ and $1 \le k \le n$.
Define $\psi_1(t, x) = t \phi(x)$ and then recursively
$\psi_k(t, x) = A \psi_{k - 1}(t, x) + t \phi(h_t^{k-1}(x))$ up to $k = n$ so that they satisfy
$$h_t^k(x) = A^k x + \psi_k(t, x).$$
Again, we have that $\lim_{x \to 0} \frac{\psi_k(t, x)}{||x||} = 0$ uniformly in $t$. Thus,
there exists a neighborhood $W'$ of $p$ such that $||\psi_k(t, x)|| \le \frac{\epsilon'}{2} ||x||$ in $W'$
for every $t \in [0,1]$ and $1 \le k \le n$. Then,
$$||x - h_t^k(x)|| = ||(I - A^k)x + t \psi_k(t, x)|| \ge ||\epsilon' x - t \frac{\epsilon'}{2} x|| > \frac{\epsilon'}{2} ||x||.$$
This means that $\{(h_t, W')\}$ is an $n$--admissible homotopy and, consequently,
$$\z_n(f, V) \e{L} \z_n(f, W') \e{H} \z_n(A, W') \e{L} \z_n(Df_p, \R^d).$$
\end{proof}

\subsection{General computation and uniqueness.}\label{subsec:general}

Theorem \ref{thm:main} is a direct consequence of the following proposition which uses the language of partial admissibility.

\begin{proposition}\label{prop:uniquezn}
For any $n \ge 1$, there exist just one function
$$\z_n: \mathcal{A}_n(U) \to 1 + t \cdot \Z [[t]] \mod 1 + t^{n+1} \cdot \Z[[t]]$$
which satisfies (N)$_n$, (M)$_n$, (H)$_n$ and (I)$_n$.
\end{proposition}

Let $f: \overline{U} \to \R^d$ be $n$--admissible. The strategy to compute $\z_n(f, U)$ is to obtain $g$
connected to $f$ by an $n$--admissible homotopy. The periodic point set of $g$ will satisfy nice transversality properties
so that the work carried out in the previous subsections will be sufficient to compute $\z_n(g, U)$.

Denote $E$ the space of continuous maps from $\overline{U}$ to $\R^d$ with the $C^0$--topology.
Fix $V \subset U$ open neighborhood of $\Per_n(f)$.
\begin{lemma}\label{lem:transversal}\leavevmode
\begin{enumerate}
\item[$(i)$] There exists a convex neighborhood $\mathcal{K}_V$ of $f$ in $E$
such that $\Per_n(g) \subset V$ for every $g \in \mathcal{K}_V$.
\item[$(ii)$] For any $m \ge 1$, in any neighborhood of $f$ in $E$
there is a $C^0$--dense set $\mathcal{D}_{m, V} \subset E$ of maps such that
for any $g \in \mathcal{D}_{m, V}$ and for every $p \in \Fix(g^m)$, $g$ is $C^1$ at $p$ and
$$\det(I - Dg^m(x)) \neq 0.$$
\end{enumerate}
\end{lemma}
\begin{proof}
Item $(i)$ is straightforward and $(ii)$ is consequence of the Transversality Theorem
(see \cite{demelo, robinson}).
\end{proof}

Take $m = (n!)^2$ and a map $g \in \mathcal{K}_V \cap \mathcal{D}_{m, V}$.
By convexity $g$ is connected to $f$ via an $n$--admissible homotopy so $\z_n(f, U) = \z_n(g, U)$.

Note that $\Per_n(g) \subset \Fix(g^m)$ and if $x$ is a $k$--periodic point of $g$ for $1 \le k \le n$
the spectrum of $Dg^k(x)$ does not contain any root of unity of order smaller or equal to $n$.
In particular, periodic points of $g$ of period at most $n$ are isolated and they satisfy the hypothesis of
Proposition \ref{prop:linearization}.
Choose, for every periodic orbit $o \subset \Per_n(g)$,
a small neighborhood $U_o \subset U$ of the orbit such that $U_o \cap U_{o'} = \emptyset$ for $o \neq o'$
and $U_o$ is composed of as many components as the period of $o$.
Multiplicativity gives
$$\z_n(g, U) = \prod \z_n(g, U_o),$$
where the product runs over all periodic orbits in $\Per_n(g)$.
Using (I) and Proposition \ref{prop:linearization},
$$\z_n(g, U_o)(t) = \z_m(g^k, U^0_o)(t^k) = \z_m(Dg^k, \R^d)(t^k),$$
where $k$ is the period of $o$, $m = \lceil \frac{n}{k} \rceil$, $U^0_o$ is one connected component of $U_o$
and equalities are understood as in axiom (I)$_n$.
The computation of $\z_n(g, U)$ and hence of $\z_n(f, U)$ is concluded using the results of Subsection \ref{subsec:hyperbolic}.

The choice of $g$ is irrelevant in the previous computation provided the ascribed values $\z_n(f, U)$ satisfy (H)$_n$.
Indeed, assume $g_0, g_1$ are two $C^1$ maps whose periodic points of period at most $n$ are isolated and satisfy
the hypothesis of the Linearization Proposition \ref{prop:linearization} so that it is easy to compute its zeta function
as was previously done. Additionally, suppose they are connected
to $f$ via $n$--admissible homotopies. If we concatenate the two homotopies we obtain an $n$--admissible
homotopy from $g_0$ to $g_1$. Axiom (H)$_n$ ensures $\z_n(g_0, U) = \z_n(g_1, U)$ and
uniqueness in Proposition \ref{prop:uniquezn} follows.

In Section \ref{sec:maszeta} it is proved that Lefschetz zeta function is a dynamical zeta function and, in particular,
it yields the required function $\z_n$ when truncating the series to order $n$.
This concludes existence and completes the proof of Proposition \ref{prop:uniquezn} and thus of Theorem \ref{thm:main}.

The previous proof involves an indirect argument which uses the known existence of a dynamical zeta function, Lefschetz zeta function.
Nonetheless, a direct argument concluding the existence of $\z_n$ from previous considerations is feasible
but beyond the scope of this article. A sketch of it would be as follows:

\begin{itemize}
\item Modify Lemma \ref{lem:transversal} to work with smooth $g_0$ and $g_1$.
\item Prove that the homotopy $\{g_t\}$ between $g_0$ and $g_1$ can be chosen smooth and satisfying
strong transversality conditions which imply that periodic points (up to period $n$) continuate except in the event of
 specific bifurcations. These bifurcations consist of two periodic points of the same period collapsing and then disappearing and take place in a 1--dimensional submanifold.
\item Check that the points occuring in the bifurcation are of types $(0,0)$ and $(0,1)$ or $(1,0)$ and $(1,1)$
 (see \ref{subsubsec:higherlinear}) so that there is no change in the zeta function.
Conclude that the approximation scheme defines $\z_n(f, U)$ unambiguously.
\end{itemize}

The second point seems the most delicate one as it would need jet transversality (cf. \cite[Proposition 9.34]{demelo})
and arguments close the ones in the proof of Kupka--Smale Theorem.
Check \cite{demelo, robinson} and references therein for more information.

\subsection{Extra dimensions}


In the case the map is the product of two maps $f \times g$ and the second factor is just a constant map
the zeta function only depends on $\z(f)$. Indeed, for hyperbolic linear maps
adding a trivial factor makes no impact in the zeta function because all new eigenvalues are 0.
Since the computation of the zeta function ultimately boils down to that of hyperbolic maps we conclude
that this observation is true in full generality.

\begin{proposition}\label{prop:extradim}
Let $(f, U)$ be an $n$--admissible pair and $U'$ an open subset of $\R^{d'}$ which contains 0. Define
$\bar{f} : U \times U' \to \R^d \times \R^{d'}$ by $\bar{f}(x, y) = (f(x), 0)$. Then,
$(\bar{f}, U \times U')$ is $n$--admissible and
$$\z_n(\bar{f}, U \times U') = \z_n(f, U).$$
Consequently, if $(f, U)$ is admissible so is $(\bar{f}, U \times U')$ and $\z(\bar{f}, U \times U') = \z(f, U)$.
\end{proposition}

\subsection{Generalized homotopy}

\begin{proposition}
Let $V$ be an open subset of $[0,1] \times U$ and
$\{(h_t, V_t)\}_{t = 0}^1$ a continuous family of maps, $V_t = V \cap (\{t\} \times U)$, such that
$(h_t, V_t)$ is $n$--admissible for every $0 \le t \le 1$. Then,
$$\z_n(h_0, V_0) = \z_n(h_1, V_1).$$
In this case, we say that $\{(h_t, V_t)\}_{t = 0}^1$ is a generalized $n$--admissible homotopy.
An analogous statement holds for (total) admissibility.
\end{proposition}

\begin{proof}
By connectedness, we only need to show that $\z_n(h_t, V_t)$ is locally constant. Fix $t_0 \in [0, 1]$
and a take a compact set $\{t_0\} \times K \subset V_{t_0}$ containing $\Per_n(h_{t_0})$ in its interior.
If $\epsilon > 0$ is small enough, $[t_0 - \epsilon, t_0 + \epsilon] \times K \subset V$ and,
furthermore, $[t_0 - \epsilon, t_0 + \epsilon] \times W \subset V$ for an open neighborhood $W$ of $K$ as well.
Additionally, we can assume that $\Per_n(h_t) \subset \{t\} \times W$ for every $|t - t_0| \le \epsilon$.
Using localization and homotopy invariance properties we obtain
$$\z_n(h_{t_0-\epsilon}, V_{t_0-\epsilon}) = \z_n(h_{t_0-\epsilon}, \{t_0 - \epsilon\} \times W) =
\z_n(h_{t_0+\epsilon}, \{t_0 + \epsilon\} \times W) = \z_n(h_{t_0+\epsilon}, V_{t_0+\epsilon}).$$
\end{proof}

\section{Around the zeta function}\label{sec:maszeta}

\subsection{Characterization through iterates}

\begin{proposition}\label{prop:iteratescharacterize}
The sequence $(\z_1(f^k, U))_{k \ge 1}$ completely determines $\z(f, U)$ and vice\-versa.
\end{proposition}
\begin{proof}
Assume first that $U$ is a small neighborhood of a hyperbolic fixed point $p$.
By Linearization Proposition \ref{prop:linearization} we can assume that $f$ is a linear map.
In \ref{subsubsec:higherlinear} we showed that there are merely four cases to consider, associated to the possible values
of the pair $(\sigma^-, \sigma^+)$.
Recall that $\sigma^{-}(\sigma^+)$ is the parity of the number of real eigenvalues of $Df_p$ smaller than -1
(greater than 1).
The results are displayed in Table (\ref{eq:table1}) below.

\begin{equation}\label{eq:table1}
\begin{array}{|c|cl|c|c|c|}
\hline
(\sigma^-, \sigma^+) & \z(f)     &													& \z_1(f)        & \z(f^2) 		  & \z_1(f^2)\\
\hline
(0,0) 							 & 1/(1 - t) & = (1 - t)^{-1} 					& 1 + t \mod t^2 & (1 - t)^{-1} & 1 + t \mod t^2\\
(0,1)							   & 1 - t		 & = (1 - t)								& 1 - t \mod t^2 & (1 - t)			& 1 - t \mod t^2\\
(1,0)								 & 1 + t		 & = (1 - t)^{-1} (1 - t^2) & 1 + t \mod t^2 & (1 - t)			& 1 - t \mod t^2\\
(1,1)								 & 1/(1 + t) & = (1 - t) (1 - t^2)^{-1} & 1 - t \mod t^2 & (1 - t)^{-1} & 1 + t \mod t^2\\
\hline
\end{array}
\end{equation}

In this case, the zeta function of the iterates of the map satisfy $\z(f^k) = \z(f)$ if $k$ is odd and
$\z(f^k) = \z(f^2)$ if $k$ is even. Note that the four possible values of $\z(f)$ are in one--to--one
correspondence with the four values of the pair $(\z_1(f), \z_1(f^2))$.

Next, assume that $\Per(f) = \Fix(f)$ is finite and every fixed point is hyperbolic.
Then, the zeta function of $f$ can be computed using multiplicativity.
Denote $a_{(0,0)},a_{(0,1)},a_{(1,0)},a_{(1,1)}$ the amount of fixed points of each of
the four types determined by $(\sigma^-,\sigma^+)$.
Then, Table (\ref{eq:table1}) yields
$$\z(f) = (1 - t)^{-a_1} (1 - t^2)^{-b_1}, \enskip \z(f^2) = (1 - t)^{-a_1 -2b_1}$$
where $a_1 = a_{(0,0)} - a_{(0,1)} + a_{(1,0)} - a_{(1,1)}$ and $b_1 = - a_{(1,0)} + a_{(1,1)}$.
Again, $\z(f^k) = \z(f)$ if $k$ is odd and $\z(f^k) = \z(f^2)$ if $k$ is even.
Identical results hold if we impose $\Per_n(f) = \Fix(f)$ and replace $\z$ by $\z_n$.

An additional bit of attention is required to analyze the case of hyperbolic periodic orbits
of period $m > 1$. Define the numbers $a^m_{(0,0)},a^m_{(0,1)},a^m_{(1,0)},a^m_{(1,1)}$ as the
amount of $m$--periodic orbits for which the pair $(\sigma^-,\sigma^+)$ is associated to $Df^m_p$, $p$ being any point in the orbit.
If $a_m = a^m_{(0,0)} - a^m_{(0,1)} + a^m_{(1,0)} - a^m_{(1,1)}$ and $b_m = - a^m_{(1,0)} + a^m_{(1,1)}$
and $d = \gcd(k, m)$ then
\begin{equation}\label{eq:zetaperiodic}
\z(f^k) =
\begin{cases}
(1 - t^{m/d})^{-da_m} (1 - t^{2m/d})^{-db_m} & \text{if }k/d \text{ is odd,} \\
(1 - t^{m/d})^{-d(a_m+2b_m)} & \text{if }k/d \text{ is even.} \\
\end{cases}
\end{equation}
Consequently, the linear term of $\z(f^k)$ vanishes excepts for the case $d = m$, that is $m | k$, so
\begin{equation}\label{eq:1zetaperiodic}
\z_1(f^k) =
\begin{cases}
1 + da_mt \mod t^2 & \text{if }m | k \text{ and } k/m \text{ is odd,} \\
1 + d(a_m+2b_m)t \mod t^2 & \text{if }m | k \text{ and } k/m \text{ is even,} \\
1 \mod t^2 & \text{otherwise.} \\
\end{cases}
\end{equation}
In particular, $\z(f) = (1 - t^{m})^{-da_m} (1 - t^{2m})^{-db_m}$ and the zeta function is determined
by the integers $a_m, b_m$, which in turn can be obtained from $\z_1(f^m), \z_1(f^{2m})$.

In the general case, recall Subsection \ref{subsec:general},
the task is reduced to examine the case in which $f$ has finitely many hyperbolic
periodic orbits of period $m$, for each $m \ge 1$. Multiplicativity yields that
$\z(f)$ is a product of factors described in (\ref{eq:zetaperiodic}),
$$\z(f) = \prod_{m \ge 1} (1 - t^{m})^{-a_m} (1 - t^{2m})^{-b_m}.$$
Thus, the exponent of each factor $(1 - t^l)$ is $-a_l$ if $l$ is odd and $-(a_l + 2b_{l/2})$ if $l$ is even.
If we denote this exponent by $e_l$, we observe that the sequence $(e_l)_{l \ge 1}$ uniquely determines $\z(f)$
and viceversa.

Similarly, $\z_1(f^k)$ is a product of factors described in (\ref{eq:1zetaperiodic}).
Note that in a product of elements of $(1 + t\cdot\Z[[t]]) / (1 + t^2 \cdot \Z[[t]])$
the linear coefficient is the sum of the linear coefficients of each of the factors.
Therefore, we obtain that
$$\z_1(f^k) = 1 + \left(\sum_{m | k} m a_m + \sum_{2m | k} 2m b_m \right) t = 1 + \left(\sum_{l | k} l e_l \right) t.$$
Consequently, the sequence $(\z_1(f^k))_{k \ge 1}$ can be obtained from $(e_l)_{l \ge 1}$
and viceversa using M\"{o}bius inversion formula. This completes the proof.
\end{proof}
\begin{remark}\label{rmk:iterates}
The previous proof shows that there is a correspondence between $(\z_1(f^k, U))^n_{k = 1}$ and $\z_n(f, U)$
via the first $n$ terms of $(e_l)_{l \ge 1}$.
\end{remark}

\subsection{Commutativity}

\begin{proposition}\label{prop:commutativity}
Let $U \subset \R^d, U' \subset \R^{d'}$ be open sets and $f: U \to \R^{d'}, g: U' \to \R^d$ be continuous maps.
Then, the composite maps
$$gf: V = U \cap f^{-1}(U') \to \R^d, \enskip \enskip fg: V' = U' \cap g^{-1}(U) \to \R^{d'}$$
have homeomorphic periodic point sets of any period. Therefore, $(gf, V)$ is $n$--admissible
iff $(fg, V')$ is $n$--admissible. Furthermore, if both are $n$--admissible then
$$\z_n(gf, V) = \z_n(fg, V').$$
Consequently, if both pairs are admissible then $\z(gf, V) = \z(fg, V')$.
\end{proposition}
\begin{proof}
If $x$ is a periodic point of $gf$ of period $k$ then $(gf)^k(x) = x$ so $f(x) = f(gf)^k(x) = (fg)^kf(x)$
hence $f(x)$ is $k$--periodic under $fg$. Conversely, if $y$ is a periodic point of $fg$ of period $k$ then
so is the point $g(y)$ under the map $gf$ and the first statement follows.

In order to address the question about zeta functions, consider the map
$F_0 : U \times U' \to \R^d \to \R^{d'}$ defined by $F_0(x, y) = ((gf)^{m-1}g(y), f(x))$.
Fix $m \le n$ and consider the homotopy $h_t(x, y) = (t(gf)^m(x) + (1-t)(gf)^{m-1}g(y), f(x))$, $0 \le t \le 1$.
Thus, if $h_t(x, y) = (x, y)$ then $y = f(x)$ and $x = t(gf)^m(x) + (1-t)g((fg)^{m-1}f(x)) = (gf)^m(x)$, so we have
$$\Fix(h_t) = \{(x, f(x)): x \in \Fix((gf)^m)\}.$$
If we assume that $((gf), V)$ is $n$--admissible then $\Fix(h_t)$ is compact and the homotopy
$\{(h_t, U \times U')\}$ is 1--admissible. Next, consider the homotopy $h'_t : U \times \R^{d'} \to \R^d \times \R^{d'}$
defined by $h'_t(x, y) = ((gf)^m(x), (1 - t)f(x))$, $0 \le t \le 1$.
Note that the restriction of $h'_0$ to $U \times U'$ is $h_1$. Now, $$\Fix(h'_t) = \{(x, tf(x)): x \in \Fix((gf)^m)\}$$
is again compact and so $\{h'_t\}$ is 1--admissible.
By concatenating both homotopies, we obtain a 1--admissible homotopy from $F_0$ to the map
$F_1(x, y) = ((gf)^m(x), 0)$. We can then use the invariance under homotopies and Proposition \ref{prop:extradim} to conclude that
$$\z_1(F_0, U \times U') = \z_1(h_1, U \times U') = \z_1(h'_0, U \times \R^{d'}) =
\z_1(F_1, U \times \R^{d'}) = \z_1((gf)^m, V).$$
Using the homotopy $h''_t(x, y) = ((gf)^{m-1}g(y), (1 - t)f(x) + t(gf)^m(y))$ and going through similar steps we can
prove that $$\z_1(F_0, U \times U') = \z_1((fg)^m, U').$$
Therefore, we have that $\z_1((gf)^m, V) = \z_1((fg)^m, V')$ for every $1 \le m \le n$.
From Remark \ref{rmk:iterates} it follows that $\z_n(gf, V) = \z_n(fg, V')$.
\end{proof}

\subsection{Extension to ENR}

The commutativity property proved in Proposition \ref{prop:commutativity} allows to extend the zeta function
to maps defined in a ENR in an identical fashion as it is done for the fixed point index (cf. Dold \cite{doldindice}).

\begin{definition}
A topological space $Y$ is a ENR (Euclidean neighborhood retract) if there exists an open subset $U$ of
an Euclidean space $\R^n$ and maps $\iota: Y \to U$, $r: U \to Y$ such that $r \iota = \id_Y$.
\end{definition}

The notion of admissibility extends similarly to this setting in the obvious way.

\begin{definition}\label{def:zetaENR}
Let $(f, V)$ be an $n$--admissible pair where $V$ is an open subset of a ENR $Y$ and $f: V \to Y$.
Let $U$ be an open subset of $\R^n$ and $\iota: Y \to U$, $r: U \to Y$ maps such that
$r \iota = \id_Y$. The zeta function of $(f, V)$ is defined as $\z_n(f, V) = \z_n(\iota f r, U)$.
Analogously, $\z(f, V)$ is defined as $\z(\iota f r, U)$ provided $(f, V)$ is admissible.
\end{definition}

It is straightforward to check the correctness of the definition. Firstly, notice that
the pair $(g = \iota f r, U)$ is $n$--admissible. Indeed, a homeomorphism between $\Fix(f^k)$ and $\Fix(g^k)$
and its inverse is given by $r$ and $\iota$. The independence of the definition from $U, r, \iota$ follows
from the commutativity: if $\iota': Y \to U'$, $r': U' \to Y'$ also satisfy $r'\iota' = \id_Y$ then
the maps $\iota r': U' \to U$ and $\iota' f r: r^{-1}(V) \subset U \to U'$ are defined in open subsets of
Euclidean spaces and by Proposition \ref{prop:commutativity} the zeta functions of their two composites are equal.
Note finally that $(\iota r') (\iota' f r) = \iota f r$ and $(\iota' f r)(\iota r') = \iota' f r'$.


\subsection{Linear term and the fixed point index}\label{subsec:z1istheindex}

Axiomatization of the topological degree was ultimately accomplished by Amann and Weiss \cite{amannweiss}.
In the Euclidean case, Ammann and Weiss theorem on the uniqueness of the degree
had been independently proved by F\"{u}hrer \cite{fuhrer}.


\begin{theorem}[Amann--Weiss, F\"{u}hrer]
There exists a unique map $\deg$ satisfying the following axioms:
\begin{itemize}
\item \textbf{Normalization.} $\deg(\id_{\overline{U}}, U) = 1$ for any open $U \subset \R^d$ containing the origin.
\item \textbf{Additivity.} For every pair of disjoint open $U_1, U_2$ subsets of $U \subset \R^d$
and $f: \overline{U} \to \R^d$ such that $0 \notin f(\overline{U} \setminus (U_1 \cup U_2))$,
$$\deg(f, U) = \deg(f_{|\overline{U_1}}, U_1) + \deg(f_{\overline{|U_2}}, U_2).$$
\item \textbf{Homotopy invariance.}
If $\{h_t: U \to \R^d\}_{t=0}^1$ is a homotopy such that $f^{-1}(0) \cap \partial U = \emptyset$ then
$$\deg(h_0, U) = \deg(h_1, U).$$
\end{itemize}
\end{theorem}

Recall that the fixed point index $i(f, U)$ of a map $f$ in $U$ is defined as the degree of $\id - f$ in $U$.
Note that $i(f, U)$ is well--defined as long as so is the degree and this is, in turn, equivalent to the pair
$(f, U)$ being 1--admissible.

For any 1--admissible pair denote $z_1(f, U)$ the linear coefficient of the series $\z_1(f, U)$.
It is straightforward to check that $z_1(\id-f, U)$ satisfies topological degree axioms, thus
$$z_1(\id - f, U) = \deg(f, U).$$
Consequently, $z_1(f, U)$ is the fixed point index of $f$ in $U$.

Lefschetz zeta function is defined for an admissible pair $(f, U)$ as
$$\exp \left( \sum_{n \ge 1} \frac{i(f^n, U)}{n} \cdot t^n \right).$$
Previous considerations imply that it satisfies axioms (N), (M) and (H) of Theorem \ref{thm:main}.

\begin{lemma}
Lefschetz zeta function satisfies (I).
\end{lemma}
\begin{proof}
By localization, $i(f^n, U) = \sum_{j = 1}^k i(f^n, U_j)$ and all the indices are zero unless $n = mk$, 
for some integer $m \ge 1$. For any $j = 1 \ldots k$, splitting $f^{mk} = f^{j-1} \circ f^{mk-j+1}$, the commutativity of
the fixed point index implies
$$i(f^{mk}, U_1) = i(f^{j-1} \circ f^{mk-j+1}, U_1) = i(f^{mk-j+1} \circ f^{j-1}, U_j) = i(f^{mk}, U_j).$$
As a consequence,
$$\sum_{n \ge 1} \frac{i(f^n, U)}{n} \cdot t^n = \sum_{m \ge 1} \frac{i(f^{mk}, U)}{mk} \cdot t^{mk} =
\sum_{m \ge 1} \frac{k \cdot i(f^{mk}, U_1)}{mk} \cdot t^{mk} = \sum_{m \ge 1} \frac{i((f^k)^m, U_1)}{m} \cdot (t^k)^m.$$
\end{proof}

In sum, we have proved the following:
\begin{theorem}\label{prop:lefschetz}
Lefschetz zeta function is the only dynamical zeta function.
\end{theorem}
\section{Symmetric products}\label{sec:productossimetricos}

\subsection{Definition}

Given a space $X$ and $n \ge 1$, the $n^{th}$ (or $n$--fold) symmetric product (or power) of $X$, denoted $SP_n(X)$,
is the quotient of $X^n$ by the action of the symmetric group $\Sigma_n$ that permutes the factors.
The projection of a point $(x_1, \ldots, x_n) \in X^n$ is denoted $[x_1, \ldots, x_n] \in SP_n(X)$.

Choosing a basepoint $x_0 \in X$ there are natural embeddings $SP_n(X) \hookrightarrow SP_{n+1}(X)$ given by
$[x_1, \ldots, x_n] \mapsto [x_0, x_1, \ldots, x_n]$. The infinite symmetric product $SP(X)$ is defined as the direct limit of
$(SP_n(X))_{n \ge 1}$ equipped with these inclusions.

There are several works in the literature proving that the topology of $X$ determines that of $SP_n(X)$ and $SP(X)$.
For instance, Macdonald \cite{macdonald} computed the Euler characteristic of $SP_n(X)$ in terms of
the Euler characteristic of $X$
and then Dold \cite{doldhomologiaSP} proved that the homology groups of $SP_n(X)$ only depend on the groups $H_q(X)$ for every
CW--complex $X$. Later, Dold and Thom \cite{doldthom} discovered that the homotopy groups of $SP(X)$ are
isomorphic to the integral homology groups of $X$.

\begin{example}[Liao \cite{liao}]\label{ex:liao}
The $n^{th}$ symmetric product of $\R^2$ has a very simple description. Indeed, after identifying $\R^2$ to $\C$, 
points in $SP_n(\C)$ can be thought of as sets of roots of monic degree--$n$ polynomials with coefficients in $\C$.
This correspondence defines a homeomorphism between $SP_n(\C)$ and $\C^n \cong \R^{2n}$.
This argument also serves to identify $SP_2(\R)$ to the space of monic degree--$2$ polynomials with real coefficients
and real roots and thus to $\{(b, c): \; b^2 \ge 4c \} \subset \R^2$. However, it is probably easier
to visualize $SP_n(\R)$ as the subset of $\R^n$ whose coordinates are in increasing order:
$\{(x_1, \ldots x_n): \; x_1 \le \ldots \le x_n\}$.
\end{example}

The following lemma will be useful in the sequel.

\begin{lemma}\label{lem:productodelaunion}[Additivity for symmetric products]
$SP_n(X \sqcup Y)$ is homeomorphic to $$SP_n(X) \sqcup (SP_{n-1}(X) \times SP_1(Y)) \sqcup \ldots
\sqcup (SP_1(X) \times SP_{n-1}(Y)) \sqcup SP_n(Y).$$
\end{lemma}
\begin{proof}
The homeomorphism $\alpha: \sqcup_{j = 0}^n SP_j(X) \times SP_{n-j}(Y) \to SP_n(X \sqcup Y)$ is
defined by $\alpha([x_1, \ldots, x_j], [y_1, \ldots\, y_{n-j}]) = [x_1, \ldots, x_j, y_1, \ldots, y_{n-j}]$.
\end{proof}

\subsection{Fixed point index}

Let \(X\)  be a ENR, \(U \subset X\) be an open set and \(f : U \rightarrow X\)  be a continuous map such that \(Per_n(f)\) is compact for some \(n \in {\mathbb N}\).
Our map \(f\) induces canonically another continuous map \(SP_n(f): SP_n(U) \rightarrow SP_n(X)\) by the formula
$SP_n(f) ([x_1,x_2, \dots, x_n])=[f(x_1),f(x_2), \dots, f(x_n))]$.

Floyd  \cite{floyd} proved that the $n^{th}$ symmetric product of an ANR (absolute neighborhood retract) is also an ANR,
see also Jaworowski \cite{jawo}. ENRs are characterized by being separable metric finite--dimensional locally compact ANRs,
see \cite{doldlibro, granas}. Since all these properties are inherited by $n^{th}$ symmetric products we conclude that
$SP_n(X)$ is a ENR.

Recall from \cite{doldindice, granas, JM} that the definition of the fixed point index in ENRs is
$i(f, V) = i(\iota f r, U)$ where we have employed the same notation as in Definition \ref{def:zetaENR}.
The fixed point index of \(SP_n(f)\) in \(SP_n(U)\), denoted $i(SP_n(f), SP_n(U)) \in \Z$, is well--defined
provided that the set of fixed points of \(SP_n(f)\) is compact in $SP_n(U)$, which is guaranteed if $\Per_n(f)$ is compact.

Most of the classical properties of the fixed point index extend in a trivial fashion to the world of symmetric products.
Only the analogue of the additivity property is not trivial and needs a proof.

\begin{proposition} {\bf Multiplicativity} \label{prop:indicesunion}
Let \(X\) be an ENR. Let $(f, U)$ be an $n$--admissible pair and $U_1, U_2$ disjoint
open subsets of $U$ such that $\Per_n(f) \cap U = \Per_n(f) \cap (U_1 \cup U_2)$ and $\Per_n(f) \cap U_1, \Per_n(f)\cap U_2$
is invariant under $f$. Then,
$$
i(SP_{n}(f), SP_n(U))= \sum_{j=0}^{n} i(SP_{j}(f), SP_j(U_1)) \cdot i(SP_{n-j}(f),SP_{n-j}(U_{2})).
$$
\end{proposition}

\begin{proof}
Firstly, observe that a fixed point of $SP_n(f)$ is an element of $SP_n(U)$ composed of periodic orbits of $f$
whose period can not be greater than $n$. Therefore, all fixed points of $SP_n(f)$ lie in $SP_n(U_1 \sqcup U_2)$
and the excision property of the fixed point index yields
$$i(SP_n(f), SP_n(U)) = i(SP_n(f), SP_n(U_1 \sqcup U_2)).$$

Lemma \ref{lem:productodelaunion} provides a suitable decomposition of $SP_n(U_1 \sqcup U_2)$. The map
$SP_n(f)$ acts on each of the terms $SP_j(U_1) \times SP_{n-j}(U_2)$ as the product of the maps
$SP_j(f): SP_j(U_1) \to SP_j(X)$ and $SP_{n-j}(f): SP_{n-j}(U_2) \to SP_{n-j}(X)$.
Using the additivity and multiplicativity properties of the fixed point index we obtain
\begin{align*}
i(SP_{n}(f), SP_n(U_1 \sqcup U_2)) &= \sum_{j = 0}^n i(SP_j(f) \times SP_{n-j}(f), SP_j(U_1) \times SP_{n-j}(U_2))\\
 &= \sum_{j=0}^{n} i(SP_{j}(f), SP_j(U_1)) \cdot i(SP_{n-j}(f),SP_{n-j}(U_{2})).\\
\end{align*}
\end{proof}

\subsection{Zeta function of symmetric products}

Given an open subset $U$ of a ENR $X$ and a map $f: U \to X$ such that $(f, U)$ is admissible, the sequence of indices
$(i(SP_n(f), SP_n(U)))_n$ is well--defined. Its generating function is
$$SP_{\infty}(f, U) = 1 + \sum_{n \ge 1} i(SP_n(f), SP_n(U)) \cdot t^n.$$

The purpose of this subsection is to show that $SP_{\infty}$ is a dynamical zeta function, i.e.,
it satisfies axioms (N), (M), (H) and (I) in Theorem \ref{thm:main}.

The following lemma can be deduced from the definition of index in $\R^d$.

\begin{lemma}
Let $V$ be an open subset of a ENR $Y$ and $c_q: V \to Y$ the constant map which maps every point to $q \in V$.
Then $i(c_q, V) = 1$.
\end{lemma}

Applying this lemma to $V = SP_n(U), Y = SP_n(X)$ and $q = [p, \ldots, p] \in V$ we obtain that
$i(SP_n(c_p), SP_n(U)) = 1$ and
$$SP_{\infty}(c_p, U) = 1 + t + t^2 + \ldots = 1/(1 - t)$$
so the normalization property (N) holds.

The multiplicativity property (M) is a direct consequence of Proposition \ref{prop:indicesunion}.
Indeed, the coefficient of $t^n$ in the product $SP_{\infty}(f, V) \cdot SP_{\infty}(f, W)$ is equal to
$$\sum_{j \ge 0}^n i(SP_j(f), SP_j(V)) \cdot i(SP_{n-j}(f), SP_{n-j}(W)) = i(SP_n(f), V \sqcup W).$$

Homotopy invariance (H) follows easily from the same property for the fixed point index
so it is only left to prove property (I).

\begin{proposition}
\textbf{Iteration.} Let $(f, U)$ be an admissible pair and assume that $U$ is the disjoint union
of $k \ge 1$, $U = U_1 \cup \ldots \cup U_{k}$, such that $f(\Per(f) \cap U_i) \subset U_{i+1}$ (indices are taken mod $k$). Then
$$SP_{\infty}(f, U)(t) = SP_{\infty}((f^k)_{|U_1}, U_1)(t^k).$$
\end{proposition}

\begin{proof}
We have to show that  $i(SP_{km}(f), SP_{km}(U)) = i(SP_m((f^k)_{U_1}), SP_m(U_1))$ for every $m \in \N$.
We can assume that \(X \subset \R^d\), \(SP_{m}(X) \subset \R^q\) and the set \(Per_{km}(f)\) is finite.

Using the excision property of the fixed point index
$$i(SP_{km}(f), SP_{km}(U))=i(SP_{km}(f), S),$$
where $S$ is an open set homeomorphic to $SP_m(U_1) \times \dots \times SP_m(U_k)$.
The result follows from the following claim.

\medskip

{\bf Claim.} Let $V \subset \R^q$ be an open set and
$(g, V)$ be an admissible pair and assume that $V$ is the disjoint union
of $k \ge 1$, $V = V_1 \cup \ldots \cup V_{k}$, such that $g(\Per(g) \cap V_i) \subset V_{i+1}$ (indices are taken mod $k$).
Let $G: V_1 \times V_2 \times \dots \times V_k \rightarrow \R^q \times \R^q \times \dots \times \R^q$
be the map defined as $G(x_1,x_2, \dots, x_k)=(g(x_k), g(x_1), g(x_2), \dots, g(x_{k-1}))$.
Then $$i(G, V_1 \times V_2 \times \dots \times V_k ) = i(g^k, V_1).$$

\medskip

{\em Proof of the Claim. }
Indeed, consider the projection $\pi_1: V_1 \times V_2 \times \dots \times V_k \rightarrow V_1$
and introduce a  map \(s: V_1\rightarrow V_1 \times V_2 \times \dots \times V_k\)  defined as
$s(x_1)=(x_1, g(x_1), g^2(x_1), \dots, g^{k-1}(x_1))$.

By the commutativity property of the fixed point index
$$i(g^k, V_1)= i(\pi_1 \circ G \circ s, V_1) = i(G \circ s \circ \pi_1, V_1 \times V_2 \times \dots \times V_k).$$

Restricted to a small neighborhood of $\Per(g)$ the homotopy
\begin{multline*}
H_t(x_1, x_2, \dots, x_k) = \\
=(g(tx_k + (1-t)g^{k-1}(x_1)), g(x_1), g(tx_2 +(1-t)g(x_1)), \dots, g(tx_{k-1} + (1-t)g^{k-2}(x_1)))
\end{multline*}

\noindent is well--defined and connects $G \circ s \circ \pi_1$ to $G$.
Moreover $\Fix(H_t)=s(\Fix(g^k) \cap V_1)$ is a compact set contained in the interior of the domain of $H_t$ so
$$i(G, V_1 \times V_2 \times \dots \times V_k) =
i(G \circ s \circ \pi_1, V_1 \times V_2 \times \dots \times V_k) = i(g^k, V_1).$$
\end{proof}

In summation, $SP_{\infty}$ satisfies all four axioms (N), (M), (H) and (I) so the proof of
Proposition \ref{prop:iszeta} is concluded. The next corollary then follows from the uniqueness
of dynamical zeta functions proved in Theorem \ref{thm:main}.

\begin{corollary}[Dold \cite{doldparma}]\label{cor:dold}
$SP_{\infty}(f,U)$ is equal to the Lefschetz zeta function
for every admissible pair $(f,U) \in \mathcal{A}(U)$,
$$1 + \sum_{n \ge 1} i(SP_n(f), SP_n(U)) \cdot t^n = \exp \left( \frac{i(f^n, U)}{n} \cdot t^n \right).$$
\end{corollary}

\section{Applications}\label{sec:applications}

\subsection{Euler characteristic of symmetric products}

The formula discovered by Macdonald for the generating function of the Euler characteristics of symmetric products
is easily obtained from Corollary \ref{cor:dold}.
\begin{corollary}[Macdonald \cite{macdonald}]\label{cor:euler}
Let $X$ be a compact ENR. Then,
$$1 + \sum_{n \ge 1} \chi(SP_n(X)) t^n = (1 - t)^{-\chi(X)}.$$
\end{corollary}
\begin{proof}
It suffices to set $f$ equal to the identity map in Corollary \ref{cor:dold} and note that by Lefschetz--Hopf theorem
$\chi(SP_n(X)) = \Lambda(SP_n(f), SP_n(X)) = i(SP_n(f), SP_n(X))$ and $i(f^n, X) = i(f, X) = \Lambda(f, X) = \chi(X)$.
\end{proof}

\subsection{Relationship between fixed point indices}

The formula in Corollary \ref{cor:dold} allows to obtain the sequences
$(i(f^n, U))_n$, $(i(SP_n(f), SP_n(U)))_n$ from one another.
More precisely, Lemma \ref{lem:despeja}, extracted from \cite[Chapter 5]{stanley}, yields the following proposition.
Notation is taken from the appendix.

\begin{proposition}
If $\mathbf{i} = (\mathbf{i}_n)_n = (i(f^n, U))_n$ and $\mathbf{s} = (\mathbf{s}_n)_n = (i(SP_n(f), SP_n(U)))_n$ then
$$
\mathbf{s}_n = \sum_{C \in \mathcal{C}_n} \frac{\pi(\mathbf{i}, C)}{\pi(C) \cdot |C|!}$$
and
$$
\mathbf{i}_n = \sum_{C \in \mathcal{C}_n} \frac{(-1)^{|C|+1}\pi(\mathbf{s}, C)}{|C|}.$$
%
%
Setting $\mathbf{s}_0 = 1$, $\mathbf{i}_0 = 0$, for $n \ge 0$
$$(n+1) \mathbf{s}_{n+1} = \sum_{j = 0}^n \mathbf{s}_{n-j} \mathbf{i}_{j+1}$$
or, equivalently,
$$\mathbf{i}_{n+1} = (n+1) \mathbf{s}_{n+1} - \sum_{j = 0}^{n-1} \mathbf{s}_{n-j} \mathbf{i}_{j+1}.$$
\end{proposition}
\begin{proof}
The first part is given by Lemma \ref{lem:despeja} and the second part can be deduced from the first one or
simply by differentiating the formula in Corollary \ref{cor:dold}.
\end{proof}

Analogous versions of these results for partially admissible pairs hold as well.

\subsection{Dold congruences}

Since the arguments used in Section \ref{sec:maszeta} are very closed to the ones contained in \cite{dolditerates}, it is natural
to obtain as a corollary of our work the main result in that article, the so--called Dold congruences.
These relations were already present in the end of the proof of Proposition \ref{prop:iteratescharacterize},
but we give a direct proof here.

\begin{theorem}[Dold \cite{dolditerates}]
Let $(f, U)$ be an admissible pair. There exists integers $\{a_k\}_{k \ge 1}$ such that, for every $n \ge 1$,
$$i(f^n, U) = \sum_{k | n} k a_k.$$
\end{theorem}
\begin{proof}
By Lemma \ref{lem:descomposicion}, there are integers $(a_k)_{k \ge 1}$ such that
$$SP_{\infty}(f, U) = \prod_{k \ge 1}(1-t^k)^{-a_k} = \exp\left( \sum_{k \ge 1}a_k \sum_{m \ge 1} \frac{t^{mk}}{m} \right).$$
This series is also equal to the Lefschetz zeta function so, after taking logarithms,
the coefficient of the term $t^n$ on both sides is
$$\frac{i(f^n, U)}{n} = \sum_{k \ge 1} a_k \frac{1}{n/k} \enskip \Leftrightarrow \enskip i(f^n, U) = \sum_{k \ge 1} k a_k.$$
\end{proof}

\subsection{Consequences of the Hopf lemma}

A non--zero fixed point index is an indicator of the existence of a fixed point. Clearly, the converse
statement is not true because there exists fixed points with zero index. For instance, the origin
under the map $f(x) = x + x^2$. However, a small perturbation of the previous dynamics makes it
fixed point free. Let us introduce the following definition.

\begin{definition}
Let $V$ be an open subset of a topological space $X$ and $f : U \to X$ a map with an isolated
fixed point $p$. Then, $p$ is called \emph{avoidable} if for every neighborhood $W$ of $p$ in $U$ there exists
a homotopy $\{f_t : U \to X\}_{t = 0}^1$ such that $f_0 = f$, $f_t = f$ outside $W$ and $\Fix(f_1) \cap W = \emptyset$.
\end{definition}

A classical argument in degree theory which dates back to H. Hopf
shows that every fixed point with zero index in a manifold $X$
is avoidable (check Hopf's Lemma in \cite{JM} for the details).
Recall that the index $i(f, p)$ of a map $f$ at a fixed point $p$ can be defined as $i(f, U)$
for $U$ an open neighborhood of $p$ such that $\Fix(f) \cap \overline{U} = \{p\}$.

As an starting example consider a repelling fixed point of a one--dimensional map. Our considerations being
local and purely topological, we may assume that the fixed point is $0 \in \R$ and the map is $f(x) = 2x$.
Clearly, this fixed point is not avoidable, indeed, it has index $-1$. Now, let us look at the $2^{nd}$ symmetric product
and the induced map $SP_2(f) : SP_2(\R) \to SP_2(\R)$ defined by $SP_2(f)[x,y] = [2x, 2y]$.
The ``double point'' $[0, 0]$ is evidently fixed under $SP_2(f)$. Surprisingly, $[0, 0]$ is avoidable.
Indeed, for any $\epsilon > 0$ define the homotopy
$$F_t : [x, y] \mapsto [2x - \lambda t, 2y + \lambda t],$$
where $\lambda = \lambda([x, y]) = \max\{\epsilon - |x| - |y|, 0\}$. A straightforward computation shows that none
of the maps $F_t$ have fixed points for $t > 0$.

Suppose now that the origin is an attracting
fixed point of $f: \R \to \R$. It then follows that $i(f^n, 0) = 1$ for every $n \ge 1$.
Corollary \ref{cor:dold} yields $i(SP_n(f), [0, \ldots, 0]) = 1$. Note that $[0, \ldots, 0]$ is an isolated
fixed point of $SP_n(f)$. In particular, $[0, 0]$ is no longer avoidable.

\begin{proposition}
Let $f : \R \to \R$ be a map for which 0 is an isolated fixed point. The following statements are equivalent:
\begin{itemize}
\item $0$ is not an attractor for $f$.
\item $[0,0]$ is avoidable for $SP_2(f)$.
\item $(0,0)$ is an avoidable fixed point of $G : \{(b, c): \; b^2 \ge 4c \} \to \{(b, c): \; b^2 \ge 4c \}$ defined by
$$G(-\alpha_1 - \alpha_2, \alpha_1 \alpha_2) = (- f(\alpha_1) - f(\alpha_2), f(\alpha_1) f(\alpha_2)).$$
\end{itemize}
\end{proposition}
\begin{proof} The third point is equivalent to the second one in view of Example \ref{ex:liao}.
For the other equivalence, note that if $0$ is neither attracting nor repelling then its index is 0.
Thus, $0$ is avoidable for $f$ and, as a consequence, $[0,0]$ and $(0,0)$ are also avoidable for $SP_2(f)$ and $G$, respectively.
\end{proof}

The fact that $SP_n(\R^2)$ is homeomorphic to $\R^{2n}$ (Example \ref{ex:liao}) implies that $n^{th}$ symmetric products of surfaces
are $2n$--dimensional manifolds. Hopf's Lemma may then be used to conclude the avoidability of
index--0 fixed points in $SP_n(M)$ for any surface $M$.

Fixed point indices of maps on surfaces have been extensively studied.
Let $f$ be an orientation--preserving homeomorphism on a surface and $p$ an isolated fixed point.
As a culmination of partial results obtained by several authors,
Le Calvez \cite{lecalvezens} proved a formula for $i(f^n, p)$ which reads as follows
$$\z(f, p) = \exp \left( \sum_{n \ge 1} \frac{i(f^n, p)}{n} \cdot t^n \right) = \frac{1}{1 - t}(1 - t^q)^{r},$$
where $\z(f, p)$ is called Lefschetz zeta function at $p$. Notice that $\z(f, p)$
is a polynomial iff $r$ is positive. In such a case, its degree is $rq - 1$.
Interestingly, it is known that if $\{p\}$ is locally maximal and neither an attractor nor a repeller
then $r$ is always positive. Furthermore, this is also the case when $f$ is area--preserving and the
sequence $(i(f^n, p))_n$ is not constant equal to 1.

Under the previous assumptions, $[p, \ldots, p]$ is fixed under $SP_n(f)$.
By Corollary \ref{cor:dold}, $i(SP_n(f), [p,\ldots,p]) = 0$ provided $r$ is positive and $n \ge rq$.
Since the symmetric product of a surface is a manifold, Hopf's Lemma can be used to conclude
that $[p, \ldots, p]$ is an avoidable fixed point of $SP_n(f)$. The number $rq$ can be thought of
as a lower bound on the number of unstable branches of the local dynamics of $f$ around $p$.
In the next paragraphs, we give, in a toy case, an explicit description of a homotopy removing the fixed point
which uses the unstable branches in a crucial way.

Let $Y = R_1 \cup R_2 \cup R_3$ be the union of three infinite sticks or half--lines with a common endpoint $p$.
Assume $p$ is a global repeller for $f : Y \to Y$.
A standard computation shows that, regardless how $f$ permutes the sticks, the zeta function of $f$ at $p$ is a
polynomial of degree 2 hence
$i(SP_3(f), [p,p,p]) = 0$ and $[p,p,p]$ is avoidable. For simplicity, in the following we assume $f$ does not
interchange the sticks.

Consider three particles moving freely within $Y$ subject to two--type of external forces: particle--center and particle--particle.
A constant force, independent of the distance, of magnitude $\epsilon$ repels every particle from $p$, center of $Y$.
There is another repelling force of much greater magnitude $M \epsilon$, $M \gg 1$, acting between every pair of particles
only when they are placed in the same stick. Two particles may occupy the same position.
The forces generate a motion in the 3--particle system.
For the evolution to be completely specified the behavior of the particles as they go through $p$ must also be prescribed.
A particle located at $p$ is set to enter an empty stick.
The same rule is applied if there are two or three particles at $p$, each of them enters a different
empty stick. In the particular case there is just one particle at $p$ and two empty sticks, the particle is frozen and does
neither interact with the others nor with the center until another particle arrives at $p$.
This bizarre behavior is imposed just to preserve symmetry.
Note, incidentally, that the evolution always forces the particles eventually lie in different sticks.

The set of possible positions of the particles can be identified to $SP_3(Y)$.
Denote $\phi_{\epsilon}$ the time--1 map of the flow generated in $SP_3(Y)$. Trivially, $\phi_0 = \id$
and the sum of the distances from each of the three points of an element of $SP_3(Y)$
 to $p$ increases under the action of $\phi_{\epsilon}$,
for any $\epsilon > 0$. The maps $SP_3(f)_{\epsilon} = SP_3(f) \circ \phi_{\epsilon}$, $0 \le \epsilon \le 1$, define
a homotopy composed of fixed point free maps except from the starting map $SP_3(f)_0 = SP_3(f)$.
This shows that $[p,p,p]$ is avoidable.
The construction can be made local just by imposing the involved forces only apply in a
neighborhood of $p$ and can also be extended to systems with more particles by suitably prescribing the behavior at $p$.
The language used to articulate this example is fortuitously related to the ``charged'' particles used by McDuff in
a paper on configuration spaces \cite{mcduff}.

Figure \ref{fig:3branches} shows the portrait of a typical dynamics around a locally fixed point $p$ of a planar homeomorphism $g$
with 3 unstable branches for which $\{p\}$ is locally maximal. For simplicity suppose the picture extends similarly to
$\R^2$ and identify the union of the 3 unstable branches with $p$ to the set $Y$.
The map $g$ leaves $Y$ invariant, denote $f = g_{|Y}$. The plane deformation
retracts to $Y$ via maps $r_t : \R^2 \to \R^2$ such that $r_0 = \id$ and $r_t(x)$ is independent of $t$ and
belongs to $Y$ iff $t \ge \dist(x, Y)$. Define $r(x) = r_t(x)$ for large $t$.
Then, for any $\epsilon \ge 0$, we can define maps $SP_3(g)_{\epsilon}: SP_3(\R^2) \to SP_3(\R^2)$ by
\begin{equation*}
SP_3(g)_{\epsilon}([x_1, x_2, x_3]) =
\begin{cases}
SP_3(g)([r_{\epsilon}(x_1), r_{\epsilon}(x_2), r_{\epsilon}(x_3)]) & \text{if } d \ge \epsilon \\
SP_3(f)_{\epsilon - d}([r(x_1), r(x_2), r(x_3)])          & \text{otherwise,}
\end{cases}
\end{equation*}
where $d = \max_i \dist(x_i, Y)$. Again all maps $SP_3(g)_{\epsilon}$ are fixed point free for $\epsilon > 0$.
This homotopy shows how the ``triple point'' $[p, p, p]$ vanishes without leaving trace in the set of fixed points,
hence proving it is avoidable.

\begin{figure}[htb]
\begin{center}
\includegraphics[scale = 0.5]{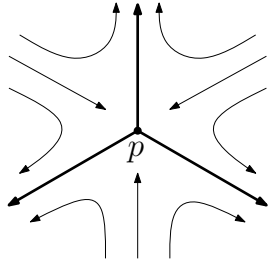}
\end{center}
\caption{Dynamics conjugate to $z \mapsto z + \bar{z}^4$ around the origin. The fixed point is locally maximal and has
3 unstable branches (painted heavier).}
\label{fig:3branches}
\end{figure}

Remind that similar constructions can just be done for the $n^{th}$ symmetric product as long as $n$ exceeds the number
of local unstable branches at the fixed point $p$. Note that in our previous
example any $n$--tuple contains a point sliding away from $p$ through every unstable branch.
The index computation suggests this property is matched by every such homotopy.

\appendix

\section{Formal power series}

Given a commutative ring with unity $R$, the ring of formal power series with coefficients in $R$, denoted $R[[t]]$, is defined as the set
of series $a(t) = \sum_{n \ge 0} a_n t^n$ where $a_n \in R$ equipped with the operations
$$(a + b)(t) := \sum_{n \ge 0} (a_n + b_n) t^n, \enskip \enskip
(a \cdot b)(t) = \sum_{n \ge 0} \left( \sum_{i = 0}^n a_i b_{n - i} \right) t^n.$$

An alternative definition involves the ring of polynomials $R[t]$ with coefficients in $R$ and
the sequence of ideals $t^n \cdot R[t]$. Then $R[[t]]$ can be identified with the inverse limit
of the sequence of rings $(R[t]/(t^n \cdot R[t]))_{n \ge 1}$ connected with the morphisms given by the natural projections.

The set $1 + t \cdot R[[t]]$ of series whose independent term is equal to 1 forms a subgroup of $R[[t]]$
with multiplication: it is closed under multiplication and the inverse of an element
$1 +  \sum_{n \ge 1} a_n t^n$ is the series $\sum_{n \ge 0} b_n t^n$ whose coefficients are defined inductively by:
$$b_0 = 1, \enskip \enskip b_n = - \sum_{i = 1}^n a_i b_{n-i}.$$
Evidently, $1 + t^n \cdot R[[t]]$ is a subgroup of $1 + t \cdot R[[t]]$ for any $n \ge 1$.

The zeta function of an admissible pair $(f, U)$ is an element of the group $1 + t \cdot \Z[[t]]$.
In a cumbersome way, it can be viewed as an inverse sequence in $\varprojlim (1 + t \cdot \Z[[t]])/(1 + t^{n+1} \cdot \Z[[t]])$.
The $n^{th}$--element of this sequence being the zeta function defined up to multiples of $t^{n+1}$, that is exactly
the zeta function $\z_n(f, U)$ of $(f, U)$ viewed as an $n$--admissible pair.

Assume henceforth that $R = \Q$ a field. We may define the following maps
\begin{equation}\label{eq:explog}
 \begin{aligned}
 \exp : t \cdot \Q[[t]] \to 1 + t \cdot \Q[[t]], \enskip \enskip \exp(x) = 1 + x + \frac{x^2}{2!} +  \frac{x^3}{3!} + \ldots\\
 \log : 1 + t \cdot \Q[[t]] \to t \cdot \Q[[t]], \enskip \enskip \log(1 - x) = - x - \frac{x^2}{2} - \frac{x^3}{3} - \ldots
 \end{aligned}
\end{equation}

\begin{lemma}
The maps $\exp, \log$ are mutually inverse isomorphisms between the groups $(t \cdot \Q[[t]], +)$ and
$(1 + t \cdot \Q[[t]], \cdot)$.
\end{lemma}

These maps are now used to define exponentiation to arbitrary formal power series
in the group $(1 + t \cdot \Q[[t]], \cdot)$. Given $1 - x \in 1 + t \cdot \Q[[t]]$ and $y \in \Q[[t]]$, define
$$(1 - x)^y := \exp(y \log(1 - x)).$$
However, we are interested only in the case where $y$ is a rational number. The definition of $\exp$ and $\log$
ensure that the exponentiation satisfies the usual properties.

\begin{lemma}\label{lem:descomposicion}
Any element of $1 + t \cdot \Q[[t]]$ is equal to the product of factors of type $(1 - t^n)$ raised to the power
of a rational number. This factorization is unique. Furthermore, all the exponents are integers iff the series has
integral coefficients.
\end{lemma}
\begin{proof}
The sequence $(\log(1 - t^n))_{n \ge 1}$ forms a basis of the free $\Q$--module $t \cdot \Q[[t]]$.
Thus, given any element $1 - x \in 1 + t \cdot \Q[[t]]$ there exists unique rational numbers $q_n$ such that
$$\log(1 - x) = \sum_{n \ge 1} q_n \log(1 - t^n)$$
or, in other words,
$$1 - x = \prod_{n \ge 1} (1 - t^n)^{q_n}.$$
For the last statement assume $a_m t^m$ is the smallest nonzero monomial of $x$ ($m \ge 1$).
If we set $y(t) = 1/(1-t^m)^{a_m} = 1 + a_m t^m + O(t^{2m})$ then all the coefficients of order 1 to $m$ of $(1 - x) \cdot y$
vanish. Thus, $q_m = a_m \in \Z$ and the argument can be carried on by induction on $m$.
\end{proof}

For a positive integer $n \ge 1$, denote $\mathcal{C}_n$ the set of compositions of $n$. Recall that a composition $C$ of $n$
is an ordered set of positive integers such that their sum is $n$. The size of $C = \{c_1,\ldots,c_k\}$
will be denoted $|C| = k$ and the product of its elements $\pi(C) = c_1 \cdot \ldots \cdot c_k$.
Given a composition $C$ and a sequence $a = (a_n)_n$ we define
$$\pi(a, C) = \prod_{j = 1}^{|C|} a_{c_j}.$$
In our setting it is simpler to work with compositions than with partitions, which do not take into account the order
of its elements.
The following lemma (cf. \cite[Chapter 5]{stanley}) is a consequence of the power series expansions (\ref{eq:explog})
of $\exp(x)$ and $\log(x)$.
\begin{lemma}\label{lem:despeja}
Let $(i_n)_n$ be a sequence of rational numbers and $z(t) = \sum_{n \ge 1} \frac{i_n}{n}t^n$.
Then, the coefficient of $t^n$ of the formal power series $\exp(z(t))$ is
$$\sum_{C \in \mathcal{C}_n} \frac{\pi(i, C)}{\pi(C) \cdot |C|!} 
.$$
Conversely, let $(s_n)_n$ be a sequence of rational numbers and $y(t) = 1 + \sum_{n \ge 1} s_n t^n$.
Then, the coefficient of $t^n$ of the formal power series $\log(y(t))$ is
$$\sum_{C \in \mathcal{C}_n} \frac{(-1)^{|C|+1}\pi(s, C)}{|C|}.$$
\end{lemma}


\begin{thebibliography}{99}


\bibitem{amannweiss}
H. Amann, S. Weiss, \emph{On the uniqueness of the topological degree}, Math. Z \textbf{130} (1973), 39--54.


\bibitem{BU}
K. Borsuk, S. Ulam, {\em On symmetric products of topological spaces}, Bull. Amer. Math. Soc. 37 (1931) 875-882.






\bibitem{Cu78}
D.W. Curtis, {\em Growth hyperspaces of Peano continua}, Trans. Amer. Math. Soc. 238, (1978) 271-283.

\bibitem{CuSc}
D.W. Curtis, R. M. Schori, {\em Hyperspaces of polyedra are Hilbert cubes}, Fund. Math. 99,  (1978) 189-197.



\bibitem{doldhomologiaSP}
A. Dold, {\em Homology of symmetric products and other functors of complexes}, Ann. of Math. 68 (1958),54--80.

\bibitem{doldindice}
A. Dold, {\em Fixed point index and fixed point theorem for Euclidean neighborhood retracts}, Topology, 4 (1965), 1--8.

\bibitem{doldlibro}
A. Dold, {\em Lectures on Algebraic Topology}, Springer, Berlin, 1972.

\bibitem{dolditerates}
A. Dold, {\em Fixed point indices of iterated maps}, Invent. Math., 74 (1983), 419--435.

\bibitem{doldparma}
A. Dold, \emph{Combinatorial and geometric fixed point theory}, Riv. Mat. Univ. Parma, 10* (1984), 23--32.

\bibitem{doldthom}
A. Dold, R. Thom, \emph{Quasifaserungen und unendliche symmetrische Produkte},
Ann. of Math. (2) \textbf{67} (1958), 239--281. 

\bibitem{granas}
J.Dugundji, A.Granas, {\em Fixed point theory}, Monografie Matematyczne, Warszawa,  PWN-Polish Scientific Publishers (1982).

\bibitem{floyd}
E. Floyd, \emph{Orbit spaces of finite transformation groups. II}, Duke Math. J. 22 (1955), 33--38.

\bibitem{fuhrer}
L. F\"{u}hrer, \emph{Theorie des Abbildungsgrades in endlichdimensionalen R\"{a}umen}, Dissertation, Freie Univ. Berlin, Berlin, 1972.

\bibitem{Gor}
L. G\'orniewicz, {\em Topological Fixed Point Theory of Multivalued Mappings}, Springer, Dordrecht (2006).




\bibitem{jawo}
J. Jaworowski, \emph{Symmetric Products of ANR's Associated with a Permutation Group}, Bull. Acad. Polon. Sci. Sér. Sci.
Math. Astronom. Phys. 20 (1972), 649--651.

\bibitem{JM}
J. Jezierski, W. Marzantowicz, {\em Homotopy methods in topological fixed and periodic point theory}, Springer, Dordrecht, 2006.





\bibitem{lecalvezens}
P. Le Calvez, {\em Dynamique des hom\'eomorphismes du plan au
voisinage d'un point fixe}. Ann. Sci. Ecole Norm. Sup. (4) 36
(2003), no. 1, 139--171.



\bibitem{liao}
S. D. Liao, \emph{On the topology of cyclic products of spheres}, Trans. Amer. Math. Soc. 77, (1954), 520--551. 

\bibitem{macdonald}
I. G. Macdonald, {\em The Poincaré polynomial of a symmetric product},
Proc. Cambridge Philos. Soc. 58 (1962) 563--568.

\bibitem{Mas}
S. Masih, {\em Fixed points of symmetric product mappings of polyhedra and metric absolute neighborhood retracts},
Fund. Math, 80 (1973) 149--156.

\bibitem{Mas1}
S. Masih, {\em On the fixed point index and the Nielsen fixed point theorem of symmetric product mappings},
Fund. Math, 102 (1979)143--158.

\bibitem{mcduff}
D. McDuff, \emph{Configuration spaces of positive and negative particles}, Topology 14 (1975), 91--107.

\bibitem{demelo}
W. de Melo, \emph{Topologia das Variedades}, IMPA, \url{http://w3.impa.br/~demelo/}.

\bibitem{Ra}
N.Rallis, {\em A fixed point index theory for symmetric product mappings}, Manuscripta Math.
44 (1983) 279--308.

\bibitem{robinson}
C. Robinson, \emph{Dynamical systems. Stability, symbolic dynamics, and chaos}.
Studies in Advanced Mathematics. CRC Press, Boca Raton, 1995.

\bibitem{RuSaII}
F.R. Ruiz del Portal, J.M. Salazar, {\em Fixed point index in hyperspaces: a Conley-type
index for discrete semidynamical systems}, J. London Math. Soc.,(2) 64 (2001) 191--204.


\bibitem{salamon}
D. Salamon, \emph{Seiberg-Witten invariants of mapping tori, symplectic fixed points, and Lefschetz numbers}. Proceedings of 6th Gökova Geometry-Topology Conference, Turkish J. Math. 23 (1999), no. 1, 117--143.

\bibitem{Sal2}
J.M. Salazar, {\em Fixed point index in symmetric products}, Trans. Amer. Math. Soc. 357 (2005) 3493--3005.



\bibitem{stanley}
R. Stanley, \emph{Enumerative combinatorics. Vol. 2,} Cambridge Stud. Adv. Math. \textbf{62},
Cambridge Univ. Press, Cambridge, 1999.



\end{thebibliography}
\end{document}